\documentclass{amsart}

\usepackage[inactive]{srcltx} 
\usepackage{amsmath, amsthm, amscd, amsfonts, amssymb, graphicx, color, extarrows}
 \newtheorem{thm}{Theorem}[section]
 \newtheorem{cor}[thm]{Corollary}
 \newtheorem{lem}[thm]{Lemma}

 \theoremstyle{definition}
 
 \theoremstyle{remark}
 
\begin{document}

\title[Linear maps characterized by special products...]
 {Linear maps characterized by special products on standard operator algebras }

\author{ Amin Barari}
\thanks{{\scriptsize
\hskip -0.4 true cm \emph{MSC(2010)}:  47L10, 47B49, 47B47.
}}

\address{Department of Mathematics, Payam Noor University, P. O. Box 19395-3697, Tehran, Iran.}

\email{aminbarari20@gmail.com; aminbarari7@gmail.com}

\address{}

\email{}

\thanks{}

\thanks{}



\date{}

\dedicatory{}

\commby{}


\begin{abstract}
Let $\mathcal{A}$ be a unital standard algebra on a complex Banach space $\mathcal{X}$ with $dim\mathcal{X}\geq 2$. We characterize the linear maps $\delta , \tau : \mathcal{A}\rightarrow B(\mathcal{X})$ satisfying $ A \tau ( B) + \delta ( A) B = 0$ whenever $A,B\in  \mathcal{A}$ are such that $AB=0$. 

\end{abstract}

\maketitle

\section{Introduction}
One of the interesting issues in mathematics is the determination of the structure of linear (additive) mappings on algebras (rings) that act through zero products in the same way as certain mappings, such as homomorphisms, derivations, centralizers, etc. For instance, see \cite{al1,bre1, bur, chb, h1, h11, gh12,gh4, h2, h3, h4,h5, h6, gh5, li, qi,xu,zhu} and references therein. Among these issues, one can point out the problem of characterizing a linear (additive) map $\delta$ from an algebra (ring) $\mathcal{A}$ into an $\mathcal{A}$-bimodule $\mathcal{M}$, which satisfy 
\begin{equation}\label{d3}
ab = 0 \Longrightarrow a \tau ( b) + \delta ( a ) b = 0 , ~~~ (a,b\in \mathcal{A}) . 
\end{equation}
where $\delta : \mathcal{A} \rightarrow \mathcal{M}$ and $\tau: \mathcal{A} \rightarrow \mathcal{M}$ are linear (additive) maps. The condition \eqref{d3} has also been studied by some authors and the mappings $\delta$ and $ \tau$ have been characterized on different algebras (rings) (see \cite{gh3, lee}). In \cite{ben}, the authors consider linear maps $\delta, \tau: \mathcal{A} \rightarrow \mathcal{M}$ satisfying \eqref{d3} and prove that if the unital algebra $\mathcal{A}$ is generated by idempotents, then $\delta$ and $\tau$ are of the form $\delta(a)=d(a)+\delta(1)a$ and $\tau(a)=d(a)+a\tau(1)$ ($a\in \mathcal{A}$), where $d: \mathcal{A} \rightarrow \mathcal{M}$ is a derivation. Also, characterizations of the maps $\delta$ and $\tau$ are given if $\mathcal{A}$ is assumed to be a triangular algebra under some constraints on the bimodule $\mathcal{M}$. In this paper, we describe the linear mappings of the standard operator algebras in a Banach space that satisfy \eqref{d3}.

\section{The main results}
Throughout this paper, all algebras and vector spaces will be over the complex field $\mathbb{C}$. Let $\mathcal{X}$ be a Banach space. We denote by $B(\mathcal{X})$ the algebra of all bounded linear operators on $\mathcal{X}$, and $F(\mathcal{X})$ denotes the algebra of all finite rank operators in $B(\mathcal{X})$. Recall that a \textit{standard operator algebra} is any subalgebra of $B(\mathcal{X})$ which contains $F(\mathcal{X})$. We shall denote the identity matrix of $B(\mathcal{X})$ by $I$. In Theorem~\ref{t21} of this article we characterize the linear maps $\delta, \tau: \mathcal{A} \rightarrow B(\mathcal{X})$ satisfying \eqref{d3}, where $\mathcal{A}$ is a unital standard operator algebra. 
\begin{thm} \label{t21}
Let $ \mathcal{X} $ be a Banach space, $dim \mathcal{X} \geq 2$, and let $ \mathcal{A} \subseteq B ( \mathcal{X} ) $ be a unital standard operator algebra. Suppose that $ \delta $ and $ \tau $ be linear maps from $ \mathcal{A} $ into $ B ( \mathcal{X} ) $ satisfying
\[ A B = 0 \Longrightarrow A \tau ( B) + \delta ( A ) B = 0 , ~~~ (A, B \in \mathcal{A}) .  \]
Then there exist $ R , S, T  \in B ( \mathcal{X} ) $ such that 
\[ \delta (A) = A S   - R A  , ~~  ~\tau ( A) = A T - S A \]
for all $ A \in \mathcal{A} $. 
\end{thm}
From Theorem~\ref{t21}, one gets the following corollary, which is already proved in \cite[Theorem 6]{jin}. So it can be said that Theorem~\ref{t21} is a generalization of \cite[Theorem 6]{jin}.
\begin{cor} \label{c2.2}
Let $ \mathcal{X} $ be a Banach space, $dim \mathcal{X} \geq 2$, and let $ \mathcal{A} \subseteq B ( \mathcal{X} ) $ be a unital standard operator algebra.  Assume that $ \delta : \mathcal{A} \to B ( \mathcal{X} ) $ is a linear map  satisfying
\[ A B = 0 \Longrightarrow A \delta ( B) + \delta ( A ) B = 0 , ~~~ (A, B \in \mathcal{A}) .  \]
Then there exist $ R , S \in B ( \mathcal{X} ) $ such that 
\[ \delta (A) = A S   - R A  \]
 for all $ A \in \mathcal{A} $ and $ R - S \in Z ( B( \mathcal{X} )) $. 
\end{cor}



\section{Proof of Theorem~\ref{t21}}
We proof Theorem~\ref{t21} through the following lemmas. 
\begin{lem} \label{l3.1}
For all $ A \in \mathcal{A} $ and $ X \in F ( \mathcal{X} ) $, we have
\[ A \tau ( X ) + \delta ( A) X = AX \tau (I) + \delta ( A X ) . \]
\end{lem}
\begin{proof}
Let $ P \in \mathcal{A} $ be an idempotent operator of rank one. Set $ Q = I - P $. Then for all $ A \in \mathcal{A} $, we obtain $ APQ = 0 $. So by assumption we have 
\[ A P \tau ( Q) + \delta ( AP ) Q = 0 . \]
Therefore,
\[ AP \tau (I) - AP \tau (P) + \delta ( AP ) - \delta ( AP ) P = 0  . \]
Hence
\begin{equation} \label{eq3}
AP \tau (I) + \delta ( AP ) = AP \tau (P) + \delta (AP) P .
\end{equation}
Since $ AQP = 0 $ ($A \in \mathcal{A}$), it follows that
\[ A Q \tau (P) + \delta ( AQ) P = 0 .\]
So
\[ A \tau (P) - A P \tau (P) + \delta (A) P - \delta ( AP ) P = 0 . \]
Consequently 
\begin{equation} \label{eq4}
A \tau (P) + \delta ( A)P = A P \tau (P) + \delta ( AP ) P. 
\end{equation}
By comparing \eqref{eq3} and \eqref{eq4}, we obtain 
\[ A \tau (P) + \delta (A) P = AP \tau (I) + \delta  (AP) . \]
By \cite[Lemma 1.1]{bur}, every element $ X \in F ( \mathcal{X} ) $ is a linear combination of rank-one idempotents, and so
\[ A \tau ( X) + \delta ( A) X = AX \tau (I) + \delta (AX)  \]
for all $ A \in \mathcal{A} $ and $ X \in F ( \mathcal{X} ) $. 
\end{proof}
\begin{lem} \label{l3.2}
For all $ A \in \mathcal{A} $ and $ X \in F ( \mathcal{X} ) $, we have
\[ \tau ( X A ) + \delta ( I) X A = X \tau (A) + \delta ( X ) A . \]
\end{lem}
\begin{proof}
Let $ P \in \mathcal{A} $ be a rank-one idempotent operator, and $ Q = I - P $. So $ PQA = 0 $ and $ QPA = 0$ for all $ A \in \mathcal{A} $. By assumption we have 
\[ P \tau ( QA ) + \delta (P) AQ = 0 \]
and
\[  Q \tau (PA) + \delta ( Q) PA  = 0 . \]
From these equations we have the followings, respectively.
\[ P \tau ( A) + \delta ( P ) A = P \tau ( PA ) + \delta (P) PA \]
and
\[ \tau ( PA) + \delta (I) PA = P \tau (PA) + \delta (P) PA . \]
Comparing these equations, we get
\[ \tau ( PA) + \delta (I) PA = P \tau (A) + \delta (P) A . \]
Now, by \cite[Lemma 1.1]{bur} we have
\[ \tau ( X A ) + \delta (I) X A = X \tau (A) + \delta (X) A  \]
for all $ A \in \mathcal{A} $ and $ X \in F ( \mathcal{X} ) $. 
\end{proof}
\begin{lem} \label{l3.3}
For all $ A,B \in \mathcal{A} $, we have
\[ \delta ( AB) = A \delta ( B) + \delta (A) B - A \delta (I) B . \]
\end{lem}
\begin{proof}
Taking $ A= I $ in Lemma~\ref{l3.1}, we find that 
\begin{equation} \label{eq5}
\delta (X) = \tau (X) - X \tau (I) + \delta (I ) X ,
\end{equation}
for all $ X \in F ( \mathcal{X} ) $. Since $  F( \mathcal{X} ) $ is an ideal in $ \mathcal{A} $, it follows from \eqref{eq5}  that
\[ \delta ( AX) = \tau ( AX ) - AX \tau (I) + \delta (I) AX  \]
for all $ A \in \mathcal{A} $ and $ X \in F ( \mathcal{X} ) $. From this equation and Lemma~\ref{l3.1}, we obtain 
\begin{equation} \label{eq6}
\tau ( AX) = A \tau (X) + \delta (A) X - \delta (I) A X 
\end{equation}
for all $ A \in \mathcal{A} $ and $ X \in F ( \mathcal{X} ) $. From \eqref{eq6}, we have
\begin{equation} \label{eq7}
\tau ( AB X) = AB \tau (X) + \delta ( AB) X - \delta (I) AB X ,
\end{equation}
for all $ A , B \in \mathcal{A} $ and $ X \in F ( \mathcal{X} ) $. On the other hand, 
\begin{align} \label{eq8}
\tau ( ABX) & = A \tau (BX) + \delta (A) B X - \delta (I) AB X  \nonumber \\
& = AB \tau (X) + A \delta (B) X - A \delta (I) B X + \delta (A) B X - \delta (I) AB X 
\end{align}
for all $ A , B \in \mathcal{A} $ and $ X \in F ( \mathcal{X} ) $. By comparing \eqref{eq7} and \eqref{eq8}, we see that
\[ \delta (AB) X = A \delta (B) X + \delta (A) B X- A \delta (I) BX    \]
for all $ A , B \in \mathcal{A} $ and $ X \in F ( \mathcal{X} ) $. Since $F ( \mathcal{X} ) $ is an essential ideal in primitive algebra $ B ( \mathcal{X} ) $, it follows that
\[ \delta ( AB ) = A \delta (B) + \delta (A) B - A \delta (I) B \]
for all $ A , B \in \mathcal{A} $.
\end{proof}
\begin{lem} \label{l3.4}
For all $ A,B \in \mathcal{A} $, we have
\[ \tau (AB) = A \tau (B) + \tau (A) B - A  \tau (I) B . \]
\end{lem}
\begin{proof}
From Lemma~\ref{l3.2} and \eqref{eq5} we conclude that
\begin{align} \label{eq9}
\tau ( X A) & = X \tau (A) + \delta (X) A - \delta (I) X A \nonumber \\
& = X \tau ( A) + \tau (x) A - X \tau (I) A , 
\end{align}
for all $ A  \in \mathcal{A} $ and $ X \in F ( \mathcal{X} ) $. Now, by using \eqref{eq9} for all $ A , B \in \mathcal{A} $ and $ X \in F ( \mathcal{X} ) $, we calculate $ \tau ( X AB ) $ in two ways and we obtain the followings.
\[ \tau ( X AB ) = X \tau (AB) - \tau (X) AB - X \tau (I) AB \]
and
\[ \tau (XAB ) = XA \tau (B) + X \tau (A) B + \tau (X) AB - X \tau (I)  AB - X A \tau (I) B . \]
Comparing these equations, we get
\[ X \tau (AB) = X A \tau (B) + X \tau (A) B - X A \tau (I) B  \]
for all $ A , B \in \mathcal{A} $ and $ X \in F ( \mathcal{X} ) $. Since $ F ( \mathcal{X} ) $ is an essential ideal in $ B ( \mathcal{X} ) $, it follows that
\[ \tau ( AB ) = A \tau (B) + \tau (A) B - A \tau (I) B \]
for all $ A , B \in \mathcal{A} $ and $ X \in F ( \mathcal{X} ) $.
\end{proof}
\begin{lem} \label{l3.5}
For all $ A \in \mathcal{A} $, we have
\[ \tau ( A) - A \tau (I) = \delta (A) - \delta (I) A .\]
\end{lem}
\begin{proof}
It follows from \eqref{eq5} and Lemma~\ref{l3.3} that 
\begin{align*}
\tau ( AX ) - AX \tau (I) & = \delta (AX ) - \delta ( I ) A X \\
& = A \delta ( X) + \delta (A) X - A \delta (I) X - \delta (I) A X \\
& = A \tau (X) - AX \tau (I) + \delta (A) X - \delta(I) AX
\end{align*}
for all $ A  \in \mathcal{A} $ and $ X \in F( \mathcal{X} ) $. On the other hand, according to the Lemma~\ref{l3.4}, for all $ A \in \mathcal{A} $ and $ X \in F ( \mathcal{X} ) $, we have
\[ \tau (AX) - A X \tau (I) = A \tau (X) + \tau ( A) X - A \tau (I) X - A X \tau (I) . \]
By comparing these equations, we find that
\[ ( \delta (A) - \delta ( I ) A ) X = ( \tau (A) - A \tau (I) ) X  \]
for all $ A \in \mathcal{A} $ and $ X \in F ( \mathcal{X} ) $. Since $F ( \mathcal{X} ) $ is an essential ideal in $ B ( \mathcal{X} ) $, it follows that 
\[  \delta (A) - \delta (I) A= \tau ( A) - A \tau (I) \]
for all $ A \in \mathcal{A} $.
\end{proof}
Now, by considering the obtained results we are ready to prove Theorem~\ref{t21}. \\ \\
\textbf{Proof of Theorem \ref{t21}:}
Define the linear map $ \Delta : \mathcal{A} \to B ( \mathcal{X} ) $ by $ \Delta (A) = \delta (A) - \delta (I) A $. It follows from Lemma~\ref{l3.3} that 
\begin{align*}
\Delta (AB) & = \delta ( AB ) - \delta (I) AB \\
& = A \delta ( B) + \delta (A) B - A \delta (I ) B - \delta ( I ) AB \\
& = A \Delta (B) + \Delta ( A) B .
\end{align*} 
So $ \Delta $ is a derivation and according to \cite[Theorem 2.5.14]{dal} there exists $ S \in \mathcal{B} ( \mathcal{X} ) $ such that $ \Delta ( A) = A S - S A $ for all $ A \in \mathcal{A} $. Set $ R = S - \delta (I) $. From the definition of $ \Delta $ we conclude that $ \delta (A) = A S - R A $ for all $ A \in \mathcal{A} $. Also, by Lemma \ref{l3.5}, we have $ \Delta (A) = \tau (A) - A \tau (I) $  for all $ A \in \mathcal{A} $. Set $ T = S + \tau (I) $. Hence $\tau ( A) = AT - S A $ for all $ A \in \mathcal{A} $. The proof of theorem is complete.



\bibliographystyle{amsplain}
\bibliography{xbib}

\begin{thebibliography}{20}

\bibitem{al1} J. Alaminos, M. Bre$\check{\textrm{s}}$ar, J. Extremera and A. R. Villena, \textit{Maps preserving zero products}, Studia Math. 193 (2009), 131--159.

\bibitem{ben}
D. Benkovi$\check{\textrm{c}}$ and M. Gra$\check{\textrm{s}}$i$\check{\textrm{c}}$, \textit{Generalized derivations on unital algebras determined by action on zero products}, Linear Algebra Appl. 445 (2014), 347--368.

\bibitem{bre1} 
M. Bre$\check{\textrm{s}}$ar, \textit{Characterizing homomorphisms, derivations and multipliers in rings with idempotents,}, Proc. Roy. Soc. Edinb. Sect. A. 137 (2007), 9--21.

 \bibitem{bur}
M. Burgos and J. S. Ortega, \textit{On mappings preserving zero products}, Linear and Multilinear Algebra, 61 (2013), 323--335.

\bibitem{chb} M.A. Chebotar, W.-F. Ke, and P.-H. Lee, \textit{Maps characterized by action on zero products}, Pacific. J. Math. 216 (2004), 217--228.


 \bibitem{dal}
H.G. Dales, \textit{Banach algebras and automatic continuity}. In: London Math. Soc. Monographs. Oxford Univ. Press, Oxford (2000).

\bibitem{h5} B. Fadaee and H. Ghahramani, \textit{Jordan left derivations at the idempotent elements on reflexive algebras}, Publicationes mathematicae-Debrecen, 92/3-4 (2018), 261--275. 



\bibitem{h1} H. Ghahramani, \textit{Additive mappings derivable at non-trivial idempotents on Banach algebras}, Linear and Multilinear algebra, 60 (2012), 725--742.

\bibitem{h11} H. Ghahramani, \textit{Zero product determined some nest algebras}, Linear algebra and its applications, 438 (2013), 303--314. 



\bibitem{gh3}
 H. Ghahramani, \textit{On rings determined by zero products}, J. Algebra and appl. 12 (2013), 1--15.
 
 \bibitem{gh12} 
H. Ghahramani, \textit{Additive maps on some operator algebras behaving like $(\alpha,\beta)$-derivations or generalized $(\alpha,\beta)$-derivations at zero-product elements}, Acta Math. Scientia, 34B(4) (2014), 1287--1300.

\bibitem{gh4}
 H. Ghahramani, \textit{On derivations and Jordan derivations through zero products}, Operator and Matrices, 4 (2014), 759--771. 


\bibitem{h2} H. Ghahramani, \textit{On Centralizers of Banach Algebras}, Bulletin of the malaysian mathematical sciences society, (1)38 (2015), 155--164.

\bibitem{h3} H. Ghahramani, \textit{Characterizing Jordan derivations of matrix rings through zero products}, Mathematica Slovaca, 65 (2015), 1277--1290.

\bibitem{h4} H. Ghahramani, \textit{Characterizing Jordan maps on triangular rings through commutative zero products}, Mediterranean Journal of Mathematics, 15 (2018), 38--48.


\bibitem{h6} H. Ghahramani and Z. Patrick Pan, \textit{Linear maps on *-algebras acting on arthogonal elements like derivations or anti-derivations}, Filomat, 32:13 (2018), 4543--4554.

\bibitem{gh5} 
H. Ghahramani, \textit{Linear maps on group algebras determined by the action of the derivations or anti-derivations on a set of orthogonal elements}, Results in Mathematics, 73 (2018), 132--146. 
 
 \bibitem{jin} 
 W. Jing, S Lu and P. Li, \textit{Characterization of derivation on some operator algebras}. Bull Austr Math Soc, 66 (2002), 227--232.
 
 
  \bibitem{lee} 
 T.-K.Lee, \textit{Generalized skew derivations characterized by acting on zero products}, Pacific J. Math. 216 (2004), 293--301.
 
 
 \bibitem{li} 
 J. Li and Z. Pan, \textit{Annihilator-preserving maps, multipliers, and derivations}, Linear Algebra Appl. 432 (2010), 5--13.
 
 
  \bibitem{qi} 
 X. F. Qi, \textit{Characterization of centralizers on rings and operator algebras}, Acta Math. Sin. Chin. Ser., 56 (2013), 459--468.
 
   \bibitem{xu} 
 W. S. Xu, R. L. An and J. C. Hou, \textit{Equivalent characterization of centralizers on $B(\mathcal{H})$}, Acta Math. Sin. English Ser. 32 (2016), 1113--1120.
 
 \bibitem{zhu} J. Zhu and C.P. Xiong, \textit{Generalized derivable mappings at zero point on some reflexive operator algebras}, Linear Algebra Appl. 397 (2005), 367--379.


\end{thebibliography}

\end{document}